\documentclass[12pt]{amsart}
\usepackage{amsmath}
\usepackage{amssymb}
\usepackage{geometry}
\usepackage{graphicx}
\usepackage{amsthm}
\usepackage{setspace}
\usepackage[final]{showkeys}

\usepackage{enumerate}
\usepackage{verbatim}
\usepackage{mathrsfs}        
\usepackage{bbm}   
 
 \def\bna{B_{\epsilon,n}(x)}
 \def\bnb{\widetilde{B}_{\epsilon,n}(x)}
 \def\bnc{\widetilde{\partial B}_{\epsilon,n}(x)}
 \def\tb{\widetilde{B}}
  \def\tm{\tilde{m}}

 \def\cn{\mathcal{C}_{\beta, n}(x)}

 \allowdisplaybreaks
\usepackage{epstopdf}
\usepackage{amsthm}
\usepackage{enumerate}

\usepackage{dsfont}

\newcommand{\diam}{{\text{diam}}}
\newcommand{\sgn}{{\text{sgn}}}

\usepackage{epstopdf}
\usepackage{amsthm}
\usepackage{enumerate}

\usepackage{dsfont}

\newtheorem{thm}{Theorem}
\newtheorem{lemma}{Lemma}
\newtheorem{proposition}{Proposition}

\title[Entry times distribution for mixing systems]{Entry times distribution for mixing systems}
  \date{\today}

\author{N Haydn}
\thanks{N Haydn, Department of Mathematics, University of Southern California,
Los Angeles, 90089-2532. E-mail: {\tt {nhaydn@usc.edu}}.} 
\author{F Yang}
\thanks{F Yang, Department of Mathematics, University of Southern California,
Los Angeles, 90089-2532. E-mail: {\tt {yang617@usc.edu}}.}

\begin{document}

\begin{abstract}
We consider the return times dynamics to Bowen balls for continuous maps on metric spaces
which have invariant probability measures with certain mixing properties.
These mixing properties are satisfied for instance by systems that allow
Young tower constructions. We show that the higher order return times to 
Bowen balls are in the limit Poisson distributed.
We also provide a general result for the asymptotic behavior of the recurrence
time for Bowen balls for ergodic systems and those with specification.
\end{abstract}

\maketitle

\section{Introduction}

Recently there has been a great interest in the statistics of return times to 
small sets and theire limiting distributions as the target sets shrink to a point
and the observation time is scaled accordingly as suggested by Kac's theorem.
Lacroix and Kupsa~\cite{KL,L02} have shown that the shrinking of the target sets 
has to done in a dynamical or geometric regular way as their examples show that
 otherwise any limiting distribution could be achieved. The first dynamical result is
 due to Doeblin~\cite{Doe} who showed that for the Gauss map higher order returns in the 
 neighbourhood of the origin are Poissonian distributed in the limit. In main stream 
 dynamics, Pitskel was the first one consider the limiting distribution for
 Axiom~A systems and showed in 1990 that for cylinders the return times are 
 Poissonian in the limit and, by an approximation argument, also for 
 metric balls for hyperbolic maps on two dimensional torii. In the successive years,
 a  sequence of  results then established that returns to cylinder sets in the limit 
 become Poissonian under increasingly general conditions
  (see e.g.~\cite{Den,Hirata2,WTW,DGS,Abadi08,AV3,HP,Kif12}.
 Similar results have recently been proven for geometric balls
 (see e.g.~\cite{CC13,HW,PS}). For dynamical balls, which are the metric equivalent of
 cylinder sets and which are are used in the 
 construction of equlibrium states and the formulation of entropy for continuous maps on 
 metric spaces, much less is known. According to a result of
 Varandas~\cite{Var} the exponential growth rate of the  recurrence time
 equals the entropy. Previously, Brin and Katok~\cite{BK} have proven a
 Shannon-McMillan-Breiman type theorem for Bowen balls. This paper builds on~\cite{HY} 
 where limiting distributions of entry and return times were
 determined together and rates of convergence we given. The principal assumption is 
 that the given invariant probability measure is $\phi$-mixing or $\alpha$-mixing.
 Although this seems  restrictive, all systems that allow a Young tower
 construction~\cite{Y2,Y3} do satisfy the $\alpha$-mixing property.
 
 In the next section we give the main results. In Section~\ref{section.recurrence}
 we prove Theorem~\ref{period} which states that for ergodic, positive entropy systems
 the minimal recurrene time grows at least linearly. 
 In Section~\ref{alpha.poisson} we prove a general result on the higher order return
 distributions for $\alpha$-mixing systems, where the return sets can be unions
 of cylinders over a countably infinite alphabet. For that purpose we use the 
 Chen-Stein method of which we give a short sketch at the beginning of the section.
 This result is then used in Section~\ref{bowen.poisson} to prove the first two 
 main theorems which in fact follow from the more general Theorem~\ref{t5}.

 \section{Main results}

Let $(X, T, \mu)$ be a measure preserving system with $T:X\to X$ continuous and
$\mu$ a $T$-invariant probability measure which we assume to be ergodic with entropy $h(\mu)>0$. 
For $A\subset X$ we denote by $W_{A,m}(x)$ the number of visits of the orbit 
$\left\{T(x),T^2(x),\dots ,T^{m}(x)\right\}$ (for some  $m\in\mathbb{N}$) to the set $A$, i.e. 
$$
W_{A,m}(x)=\sum\limits_{j=1}^{m}\chi_A(T^j(x))
$$
where $\chi_A$ is the characteristic function of the set $A$, i.e.\
$\chi_A(x)=1$ if $x\in A$ and $\chi_A(x)=0$ otherwise.  The purpose of this 
paper is to get results on the distributions of $W$ in the case when the return set $A$ is
 a Bowen ball and the cutoff values $m$ for the length of the orbits are scaled by
 the measures of the return set.
 Clearly $W_{A,m}(x)=0$ if the entry/return time $\tau_A(x)$ is larger than $m$,
 where $\tau_A(x)=\min\{j\ge1: T^jx\in A\}$.

Let $\mathcal{A}$ be a finite measurable partition and denote 
by $\mathcal{A}^n=\bigvee_{j=0}^{n-1}T^{-1}\mathcal{A}$ the $n$th join ($n$-cylinders).
We assume that $\mathcal{A}$ is generating, i.e.\ that $\mathcal{A}^\infty$ consists of 
single points. For a set $Y\subset X$ we shall use the notation 
$A_n(Y)=\bigcup_{A\in\mathcal{A}^n, \,A\cap Y\not=\varnothing}A$ as the smallest union
of $n$-cylinders that approximates $Y$ from the outside. In particular $A_n(x)$ denotes
the $n$-cylinder that contains $x$.

We shall require that the measure $\mu$ have some mixing property with respect to this partition $\mathcal{A}$. To be more precise, we say that $\mu$ is {\em $\phi$-mixing} if
$$
|\mu(A\cap T^{-n-k}B) -\mu(A)\mu(B)| \le \phi(k)\mu(B)
$$
for all $A \in \sigma(\mathcal{A}^n)$, $B \in \sigma(\bigcup_j\mathcal{A}^j)$, where $\phi(k)$ is a decreasing function that converges to $0$. Similarly, we say that $\mu$ is {\em $\alpha$-mixing} if 
$$
|\mu(A\cap T^{-n-k}B) -\mu(A)\mu(B)| \le \alpha(k)
$$
for some decreasing function $\alpha(k)$ that converges to $0$.

We will also require some regularity of the measure. For $x \in X, 0<\delta<\epsilon$ we define
$$
\psi(\epsilon, \delta, x) = \frac{\mu(B(x,\epsilon+\delta))-\mu(B(x,\epsilon -\delta))}{\mu(B(	x,\epsilon))}
$$
as in \cite{HY}. $\psi$ measures the proportion of the measure of the annulus $B(x,\epsilon+\delta)\setminus B(x,\epsilon -\delta)$ to the ball $B(x, \epsilon)$.

For $\epsilon>0$ and $n \in \mathbb{N}$ we define the {\em $(\epsilon, n)$-Bowen ball} as usual:
$$
\bna = \{y: \sup_{0 \le k < n} d(T^kx, T^ky)<\epsilon\}.
$$
Bowen balls have the property that they capture the local dynamics in metric 
spaces and are used to define entropy, pressure and prove the existence of 
equilibrium states for given potential functions (see e.g.~\cite{Wal}). In many ways
Bowen balls play on metric spaces the r\^ole that cylinder sets play in symbolic systems. For instance,
according to Brin and Katok~\cite{BK} one has the metric analogue of the theorem of 
Shannon-McMillan-Breimann:
$$
\lim_{\epsilon\to0}\lim_{n\to\infty}\frac1n|\log\mu(\bna)|=h(\mu)
$$
for $\mu$-almost every $x$ provided $\mu$ is  ergodic.
Similarly, Varandas~\cite{Var} provided us with the metric equivalent of Ornstein-Weiss' formula
for the {\em recurrence time} $R_{\epsilon,n}(x)=\min\{j\ge1: T^jx\in\bna(x)\}$, according
to which 
$$
\lim_{\epsilon\to0}\lim_{n\to\infty}\frac1n\log R_{\epsilon,n}(x)=h(\mu)
$$
$\mu$ almost everywhere for ergodic $\mu$.
In a previous paper~\cite{HY} we studied the distribution of the first entry and return times.
Here we take up the subject of higher order returns. We have the following result.

\begin{thm}\label{t1} Assume that the invariant measure $\mu$ is $\phi$-mixing where
$\phi(n) =\mathcal{O}( \frac{1}{n^{2+\kappa}})$ and 
 $\mbox{diam}(\mathcal{A}^n)=\mathcal{O}(\gamma^{n^\xi})$ for some $\gamma < 1$, $\xi \le 1$. 
 Moreover assume that $\mu$ satisfies the following regularity condition
$$
\psi(\epsilon, \delta, x) \le \frac{C_{\epsilon}}{|\log \delta|^{\zeta}}
$$
for some $\zeta>1/\xi$ , $C_{\epsilon}>0$ independent of $x$. Put $m = \frac{t}{\mu(\bna)}$.

Then there exists $\epsilon_0>0$ so that for every $\epsilon<\epsilon_0$ we have
$$
\lim\limits_{n \rightarrow \infty}  \mathbb{P}\big(W_{\bna, m} =k \big) 
=e^{-t}\frac{t^k}{k!}
$$ 
almost surely.
\end{thm}

If the measure has better regularity then we can relax the condition on the diameter of cylinders
and obtain the following statement:

 \begin{thm}\label{t1'}
Assume that there exist constants $\alpha, \kappa, \xi>0$ satisfying $\alpha \xi > 1$, such that $\diam(\mathcal{A}^n) = \mathcal{O}(n^{-\alpha})$, $\phi(n) = \mathcal{O} (n^{-(2 + \kappa)})$ and
$$
\psi(\epsilon, \delta, x) \le C_\epsilon \delta^\xi
$$ 
for some constant $C_\epsilon$ independent of $\delta$ and $x$. 

Then there exists $\epsilon_0>0$ so that for every $\epsilon<\epsilon_0$ we have
$$
\lim\limits_{n \rightarrow \infty}  \mathbb{P}\big(W_{\bna, m} =k \big) 
=e^{-t}\frac{t^k}{k!}
$$ 
almost surely, where $m = \frac{t}{\mu(\bna)}$.
\end{thm} 

The proof of these two theorems is in Section~\ref{bowen.poisson}. In the proof of these theorems we need some estimate on the minimum return time of points in $\bna$, which is in the next section.

\section{Recurrence time for dynamical balls}\label{section.recurrence}

For a set $A\subset X$ we have the  first hitting time of a point $x$ given by 
$\tau_A(x) = \min\{k>0: T^k(x) \in A\}$.
The {\em period} of the set $A$ is then given by 
$$
\tau(A) = \min\{k>0: T^{-k}(A) \cap A \ne \emptyset \}
$$
which evidently equals $\tau(A) = \min_{x\in A}\tau_A(x)$.
The statement of the following theorem is well known for cylinder sets~\cite{STV}
and will here be proven for Bowen balls.

\begin{thm}\label{period}Assume that $\mu$ is ergodic with entropy $h(\mu)>0$. \\
(i) Then for almost every $x \in X$ 
$$
\lim\limits_{\epsilon \to 0} \liminf\limits_{n \to \infty} \frac{\tau(\bna)}{n} \ge 1.
$$
(ii) If, moroever, the map $T$ has specification, then 
$$
 \limsup\limits_{n \to \infty} \frac{\tau(\bna)}{n} \le 1
$$
for all $\epsilon$ small enough.
\end{thm}

Let us recall that a map $T:X\to X$ has {\em specification} if 
for every $\epsilon>0$ there exists a separation time $K(\epsilon)$ so that any 
 two (in fact arbitrarily many) orbit segments $T^jx, j=0,1,\dots,n_x$ and $T^jy, j=0,1,\dots,n_y$ 
 can be $\epsilon$-shadowed by an actual orbit, that is there exist a point $z\in X$ and 
 $m\le K$ such that $d(T^jz,T^jx)<\epsilon$ for $j=0,1,\dots,n_x$ and 
 $d(T^{n_x+1+m+j}z,T^jy)<\epsilon$ for $j=0,1,\dots,n_y$.

In order to prove the lower bound~(i) we need the following lemma.

\begin{lemma}\label{cluster of cylinders}\cite{BK}
Let $\mathcal{A}$ be a finite generating partition with $\mu(\partial\mathcal{A}) = 0$. Then for all $\delta>0$ there exist $N>0$ and a set $D_N$ with $\mu(D_N)>1-\delta$, such that 
$$
\left|\{ A \in \mathcal{A}^n: A \cap\bna \ne \emptyset\} \right| \le e^{\delta n}\qquad \forall\;x\in D_N
$$
for all $\epsilon$ small enough and $n \ge N$.
\end{lemma}

\begin{proof} 
For all $\epsilon > 0$, define $U_{\epsilon}(\mathcal{A}) = \bigcup\limits_{A \in \mathcal{A}} U_{\epsilon} (A)$ where 
\[
U_{\epsilon}(A) = \{x \in A: \text{there exist } y \in X\setminus A \text{ with } d(x,y) < \varepsilon  \}
=A\cap B(X\setminus A,\epsilon).
\]
Since $\bigcap\limits_{\epsilon>0} U_{\epsilon}(\mathcal{A}) = \partial \mathcal(A)$, we have $\lim\limits_{\epsilon \to 0} \mu(U_{\epsilon}(\mathcal{A})) = 0 $, and thus for every \(\beta > 0\), 
there exists  \(\epsilon_0\) small enough so that 
$$
\mu(U_\epsilon(\mathcal{A}))< \beta / 2, \text{ for all } \epsilon < \epsilon_0.
$$
By the Birkhoff ergodic theorem,
$$
\lim\limits_{n \to \infty} \frac{1}{n} \sum\limits_{k=0}^{n-1} \chi_{U_\epsilon(\mathcal{A})}(T^k(x)) < \beta / 2, \text{ for a.e } x \in X.
$$
Take $N_1$ so that the set $D_N$ defined by
$$
D_N = \{x \in X: \frac{1}{n} \sum\limits_{k=0}^{n-1} \chi_{U_\epsilon(\mathcal{A})}(T^k(x)) < \beta  \quad \forall n \geq N     \}
$$
satisfies 
$$
\mu(D_N) > 1- \delta , \text{ for all } N > N_1.
$$

Every $n$-cylinder $A_n(x)$ is identified by the $n$-word $x_0 x_1 \cdots x_{n-1}$ 
where $x_k \in\mathcal{A} $. We call this word the $(\mathcal{A}, n)$-name of $A_n(x)$. For all $y \in \bna$ and \(0 \leq k \leq n-1 \), either $T^k(y) \in A_1(T^k(x)) $, or $T^k(x) \in U_\epsilon (\mathcal{A})$. Now let us note that for all $x$ in $D_N$, the frequency of the latter possibility
(i.e.\ $T^k(x) \in U_\epsilon (\mathcal{A})$) is less than $\beta$. In other words,
 $d^H_n(x,y) < \beta$ for all $y \in \bna$, $x \in D_N $ and $n > N$, where $d_n^H$ 
 is the Hamming distance given by 
 $d_n^H(x,y)=\frac1n\sum_{k=0}^{n-1}(1-\delta_{x_k,y_k})$ with $\delta$ denoting the 
 Kronecker symbol.

If we denote  $\cn = \{y: d^H_n (x, y) < \beta \}$ the cluster of $n$-cylinders centred at $x$ then
$$
\bna \subset A_n(\bna) \subset \cn, \quad\forall x \in D_N, n > N.
$$
Since  \(d^H_n(x,y) = 0\) if the points \(x, y\) lie in the same element of $\mathcal{A}^n$, 
$\cn$ is a union of at most $\lambda_n$ elements in $\mathcal{A}^n$, where $\lambda_n$ can 
be estimated by 
\begin{equation}\label{lambda.estimate}
\lambda_n \leq \sum\limits_{m=0}^{[n \eta]}|\mathcal{A}|^m \binom{n}{m}.
\end{equation}
Using Stirling's formula, it is easy to show that 
$$
\limsup\limits_{n \to \infty}\frac{\log\lambda_n}{n} \leq \beta \log|\mathcal{A}| - \beta\log\beta - (1-\beta)\log(1-\beta).
$$
The right-hand-side converges to 0 as \(\beta\) approaches \(0\). For any given $\delta>0$ we can take $\beta$ small enough such that $\lambda_n \le e^{\delta n}$ for all $n\ge N$ for some 
large enough $N$. In particular 
$$
\big|\{A \in \mathcal{A}^n, A \cap \bna \ne \emptyset\}\big| \le e^{\delta n}.
$$
\end{proof}


\begin{proof}[Proof of Theorem~\ref{period}]
Let $\delta>0$ and $D_N$, $N$ as in Lemma~\ref{cluster of cylinders}.
Then for all $x \in D_N$ we have $\bna \subset \cn$,
where $\cn=\{y:d_n^H(x,y)<\beta\}$ with $\beta>0$ being chosen below. Hence
$$
\tau(\bna) \ge \tau(\cn).
$$

For arbitrary $\eta<1$, fix $\zeta < \frac{1-\eta}{8}h$ small and let  
$E_N = \{x: e^{-(h + \zeta)n} \le \mu(\mathcal{A}^n(x)) \le e^{-(h - \zeta)n} \text{ for all } n \ge N \}$. 
By the Theorem of Shannon-McMillan-Breiman,  we can take $N$ large such that 
$\mu(E_N) \ge 1 - \delta$. Set $G_N = D_N \cap E_N$, we have $\mu(G_N) \ge 1 - 2\delta$. 
For a large enough constant $c_1$ (depending on $N$) we achieve that $c_1^{-1}e^{-(h + \zeta)n} \le \mu(\mathcal{A}^n(x)) \le c_1e^{-(h - \zeta)n}$ hold for all $n>0$. 

Define 
$$
B_n = \{x\in G_N: \tau(\bna)< \eta n \}
$$
and
$$
\tilde B_n = \{x\in G_N: \tau(\cn) < \eta n\}.
$$
Clearly $B_n \subset \tilde B_n$ for all $n \le N$.

If we put $R^n(k) = \{x: \tau(\cn) = k\}$ ($k \le [\eta n]$) then 
$\tilde B_n = \bigcup_{k=1}^{\eta n}R^n(k)$ (disjoint union).
In other words, if $x \in R^n(k)$ then $T^j\cn\cap\cn=\varnothing$ for $j=1,\dots,k-1$
 and  there exists some $y \in \cn$ such that $T^k(y) \in \cn$.
Hence we have
 $d^H_n(y, T^k y) \le d_n^H(y,x)+d_n^H(x,T^ky)\le 2 \beta$. Set $\tilde R^n(k) = \{y: d^H_n(y, T^ky) \le 2 \beta\}$ and we obtain
$$
R^n(k) \subset \{x : \text{ there exist } y \in \tilde R^n(k) \text{ such that }d^H_n(x,y) \le \beta  \}.
$$
First we estimate $\mu(\tilde R^n(k))$. For every $y \in \tilde R^n(k)$, let 
$$
A_n(y) = (y_1\ldots y_k y_{k+1} \ldots y_{2k}\ldots y_{mk+1} \ldots y_n)
$$
with $y_i \in \mathcal{A}$, $m = [\frac{n}{k}]$, then 
$$
A_n(T^ky) = (y_{k+1}\ldots y_{2k} y_{2k+1} \ldots y_{3k}\ldots y_{(m+1)k+1} \ldots y_{n+k}).
$$
Let $g_i = \sum\limits_{j=ik+1}^{(i+1)k}(1 - \delta_{y_j, y_{j+k}})$ for $i=1, 2, \ldots, m$, 
where $\delta_{a,b}$ is the standard Kronecker symbol. That is $g_i$ is the number of 
coordinates on which  $y_{ik+1} \ldots y_{(i+1)k}$ and $y_{(i+1)k+1} \ldots y_{(i+2)k}$ differ.
Obviously  $g_i \le k$ and also  $\sum\limits_{i=1}^{m} g_i \le 2\beta n$ as $y\in \tilde R^n(k)$.

For given $(g_1,g_2,\dots,g_m)$ and given $k$-word $y_1y_2\cdots y_k$, the total number of 
 $n$-cylinders $A_n(y)$ that lie in the given $A_k(y)=(y_1y_2\cdots y_k)$ and for which
 $y\in\tilde{R}^n(k) $ is bounded from above by 
\begin{align*}
a_{n, y_1, \ldots, y_k, g_1, \ldots, g_m} 
\le & \binom{k}{g_1} |\mathcal{A}|^{g_1} \binom{k}{g_2} |\mathcal{A}|^{g_2} \ldots \binom{k}{h_m}|\mathcal{A}|^{g_m}\\
\le & \binom{n}{2 \beta n} |\mathcal{A}|^{2 \beta n}.
\end{align*}
To simplify notation, we abbreviate the LHS to $a_n$. By Stirling's formula 
\begin{align*}
\frac{\log a_n}{n} \le 2\beta \log |\mathcal{A}| - (1-2\beta)\log(1-2\beta) - 2\beta \log 2\beta \to 0
\end{align*}
as $\beta \to 0$. We can take $\beta$ small such that $a_n \le e^{\delta n}$ where $\delta>0$
 is as above.

Denote by $b_{n,k}$ the total number of such possible $(g_1,\dots,g_m)\in\{1,2,\dots,k\}^m$.
Then
$$
b_{n,k} = \sum\limits_{j=0}^{[2\beta n]} \binom{j+m-1}{m-1} = \binom{[2 \beta n]+m}{m} = \binom{[2\beta n] + \frac{n}{k}}{\frac{n}{k}}
$$
which, again by Stirling's formula, can be bound as follows:
\begin{align*}
\frac{\log b_{n,k}}{n} \le   f(2\beta + \frac{1}{k} ) - f(\frac{1}{k}) - f(2\beta),
\end{align*}
where we put $f(x) = x\log x$.
Since $f(x) \to 0$ as $x \to 0$ and $f(x)$ is uniformly continuous on $(0,2]$, we have 
$\lim\limits_{\beta \to 0}\frac{\log b_{n,k}}{n} =0$ and in particular
$$
b_{n,k} \le e^{\delta n}
$$
if we only take $\beta$ small enough.

All the above estimates combined now yield:
\begin{align*}
\mu(\tilde R^n(k)) \le & \sum\limits_{y \in \tilde R^n(k)} \mu(A_n(y))\\
\le & \sum\limits_{A_k(y), y \in \tilde R^n(k)} \sum\limits_{g_1, \ldots, g_m} a_n  c_1e^{-(h - \zeta)n}\\
\le & \sum\limits_{A_k(y), y \in \tilde R^n(k)} b_{n,k} c_1e^{-(h - \zeta-\delta)n}\\
\le & \sum\limits_{A_k(y), y \in \tilde R^n(k)} c_1e^{-(h - \zeta- 2\delta)n}.
\end{align*}
For $y \in \tilde R^n(k) \subset G_N$, we have $c_1^{-1}e^{-(h + \zeta)k} \le \mu(A_k(y))$, hence 
$1 \le c_1e^{(h + \zeta)k} \mu(A_k(y))$. Therefore
\begin{align*}
\mu(\tilde R^n(k)) \le & \sum\limits_{A_k(y), y \in \tilde R^n(k)} c_1^2e^{-(h - \zeta- 2\delta)n} e^{(h + \zeta)k} \mu(A_k(y))\\
\le &  c_1^2e^{-(h - \zeta- 2\delta)n+(h + \zeta)k}.
\end{align*}
Since $c_1^{-1}e^{-(h + \zeta)n} \le \mu(A_n(x)) \le c_1e^{-(h - \zeta)n}$ for every 
$n$-cylinder in $G_N$, $\tilde R^n(k)$ can be covered by at most 
$c_1^3e^{-(h - \zeta- 2\delta)n+(h + \zeta)k +(h + \zeta)n }$ many $n$-cylinders. 
Since $ R^n(k)$ is contained in the $\beta$-neighbourhood of $\tilde R^n(k)$ (under the
Hamming metric $d^H_n$), and every $\beta$-neighbourhood of an $n$-cylinder contains
 at most $\lambda_n < e^{\delta n}$ many $n$-cylinders according to~\eqref{lambda.estimate},
 the total number of $n$-cylinders that intersects 
 $R^n(k)$ is bounded from above by 
$$
\lambda_n c_1^3e^{-(h - \zeta- 2\delta)n+(h + \zeta)k +(h + \zeta)n } \le c_1^3e^{(  2\zeta + 3\delta)n+(h + \zeta)k }.
$$
Therefore,
$$
\mu(R^n(k)) \le  c_1^3e^{(  2\zeta + 3\delta)n+(h + \zeta)k } c_1e^{-(h - \zeta)n}
\le  c_1^4e^{(-h+3\zeta + 3\delta)n + (h + \zeta)k}.
$$

Summing over $k$, we finally obtain
\begin{align*}
\mu(\tilde B_n) \le &\sum\limits_{k=1}^{\eta n} \mu(R^n(k))\\
\le & \sum\limits_{k=1}^{\eta n} c_1^4e^{(-h+3\zeta + 3\delta)n + (h + \zeta)k}\\
\le & c_2 e^{(-h+3\zeta + 3\delta)n + (h + \zeta)\eta n}\\
\le & c_2 e^{(-(1-\eta)h+4\zeta + 3\delta)n }.
\end{align*}

Since $B_n \subset \tilde B_n$ for all $n \ge N$, we have 
$$
\sum_n\mu(B_n) \le N+\sum_{n>N}\mu(\tilde B_n) \le N+  \sum_{n>N}c_2e^{(-(1-\eta)h+4\zeta + 3\delta)n }.
$$
We can choose $\delta < \frac{1-\eta}{8}$ and $\zeta < \frac{1-\eta}{8}h$, hence 
$$
-(1-\eta)h+4\zeta + 3\delta \le -\frac{1-\eta}{8}h < 0.
$$
 Therefore $\sum_n\mu(B_n) < \infty$. By the Borel-Contelli lemma, for almost every $x\in G_N$ we have $\liminf\limits_n \frac{\tau(\bna)}{n}\ge \eta$. Since $\eta<1$ is arbitrary, the lower bound~(i) of
 the theorem follows.
 
 In order to get the upper bound~(ii) for a map with specification let $K(\epsilon)$ be
 the separation time. Then there exists a point $z\in\bna$
 and an $m\le K$ so that $T^{n+m}\in\bna$. Hence $\tau(\bna)\le n+K$
 and consequently $\lim_{n\to\infty}\frac1n\tau(\bna)\le1$.
 \end{proof}

\section{$\alpha$-mixing system have Poisson distributed return times for
unions of cylinders}\label{alpha.poisson}

This section is on the return times to sets that are unions of cylinders, where the 
underlying partition $\mathcal{A}$ is allowed to be countably infinite.
Recall that $W_{A,m}(x)$ is the number of visits of the orbit 
$\left\{T(x),T^2(x),\dots ,T^{m}(x)\right\}$ to the set $A$, i.e. 
$$
W_{A,m}(x)=\sum\limits_{j=1}^{m}\chi_A(T^j(x)).
$$
We then have the following result. 

\begin{thm}\label{Poisson_cylinder} Let $\mu$ be $\alpha$-mixing w.r.t.\ a finite or countably infinite partition
$\mathcal{A}$ and let $A\in\sigma(\mathcal{A}^n)$. As before, let $\tau(A)$ be the period of 
$A$. For any $t>0$, let $m = \frac{t}{\mu(A)}$ and denote by $\nu_t$ the Poisson
measure on $\mathbb{N}_0$ with parameter $t$. 
Then there exists a constant $C_1$ so that for every set $E\subset\mathbb{N}_0$
\begin{eqnarray*}
|\mathbb{P}(W_{A, m}\in E)-\nu_t(E)| \hspace{-1cm}&&\\ &\le&C_1\min_{\tau(A)<\Delta<m}\left(\frac{\alpha(\Delta)}{\mu(A)}+\Delta\mu(A)+\mathbb{P}_A(\tau_A \le \Delta)\right
)(t+\log m).
\end{eqnarray*}

\end{thm}

For similar result see~\cite{Abadi08,AV1,AV3} where the Poisson distribution for $\phi$-mixing
measures was shown for single cylinders centred at a generic point.
To prove this theorem we use the Chen-Stein method similar to~\cite{HP}
where it was laid out in more detail than we do here although we shall proceed 
to give a  summary of the procedure.

Let $\nu$ be a probability measure on $\mathbb{N}_0$ (equipped with the 
power $\sigma$-algebra $\mathscr{B}_{\mathbb{N}_0}$).
If we denote by $\mathcal{F}$ the set of all real-valued functions on  $\mathbb{N}_0$, then the Stein 
operator $\mathcal{S}:\mathcal{F}\rightarrow\mathcal{F}$ is defined by
\begin{equation}\label{steinoperator}
\mathcal{S}f(k)=tf(k+1)-kf(k),\quad\text{ }  \forall k\in \mathbb{N}_0.
\end{equation}
Denote by $\nu_t$ the Poisson-distribution measure with mean $t$, i.e.\
$\mathbb{P}_{\nu_t}(\{k\})=\frac{e^{-t}t^k}{k!}$  $\forall k\in \mathbb{N}_0$ then the Stein equation 
\begin{equation}\label{steineq}
\mathcal{S}f=h-\int_{\mathbb{N}_0}h\,d\nu_t
\end{equation}
has a solution $f$ for each 
$\nu_t$-integrable $h\in\mathcal{F}$ (see~\cite{BC1}). 
The solution $f$ is unique except for $f(0)$, which can be chosen arbitrarily\footnote{
 $f$ can be computed recursively from the Stein equation:
$$
f(k)=\frac{(k-1)!}{t^k}\sum_{i=0}^{k-1} \left( h(i)-\mu_0(h)\right)\frac{t^i}{i!}
=-\frac{(k-1)!}{t^k}\sum_{i=k}^{\infty} \left( h(i)-\mu_0(h)\right)\frac{t^i}{i!} , 
\quad\text{} \forall k\in\mathbb{N}.
$$}.
In particular,  if $h:\mathbb{N}_0\rightarrow\mathbb{R}$ is bounded then so is the associated Stein solution $f$.
A probability measure $\nu$ on $(\mathbb{N}_0,\mathscr{B}_{\mathbb{N}_0})$ is Poisson 
(with parameter $t$) if and only if~\cite{BC1}
$\int_{\mathbb{N}_0}\mathcal{S}f\,d\nu=0$ for all bounded functions
$f:\mathbb{N}_0\rightarrow\mathbb{R}$.
The total variation distance of a probability measure $\nu$  from
the  Poisson distribution $\nu_t$ can then be estimated as follows:
\begin{equation}\label{finalformstein}
|\nu(E)-\nu_t(E)|=\left | \int_{\mathbb{N}_0}\mathcal{S}f\,d\nu\right |
=\left | \int_{\mathbb{N}_0}\left (t f(k+1)-kf(k)\right)d\nu\right |
\end{equation}
where $E\subset\mathbb{N}_0$ and  $f$ is the Stein solution that corresponds to the 
indicator function $\chi_E$. 
 The following lemma on the function $f$ associated to characteristic functions was proven in~\cite{HP}.

\begin{lemma}\label{logsum}
For the Poisson distribution $\mu_0$, the Stein solution of the Stein equation ($\ref{steineq}$) that corresponds to the indicator function $h=\chi_E$, with $E\subset\mathbb{N}_0$, satisfies
\begin{equation}\label{estimatessteinsolution}
\left|f_{\chi_E}(k)\right|\le
\begin{cases}
    1  \quad&\text{ if } k \le t\\
     \frac{2+t}{k} \quad&\text{ if } k >t\;.
\end{cases}
\end{equation} 
In particular
\begin{eqnarray}
\sum\limits_{k=1}^{m}\left|f_{\chi_E}(k)\right|
&\le& \begin{cases} m \quad&\text{if } m\le t \\
t+(2+t)\log\frac{m}{t} \quad&\text{if } m>t\;.
\end{cases}
\end{eqnarray}

\end{lemma}

\subsection{Return times distribution}

\begin{proof}[Proof of Theorem~\ref{Poisson_cylinder} ]

The Poisson parameter $t$ is the expected value of $W_{A,m}$ 
which implies $t=\sum\limits_{i=1}^m\mu\left(\chi_AT^i\right)=m\mu(A)$,
where $\mu(T^{-i}A)=\mu(A)$ by invariance.
If $h=\chi_E$ with $E\subset\mathbb{N}_0$ an arbitrary subset of the positive integers,
then we obtain from~\eqref{finalformstein} and~\eqref{steinoperator}
$$
\left |\nu(\mathcal{S}f)\right |
=\left | \nu( h)-\nu_t(h)\right | =\left|\mathbb{P}(W_{A,m}\in E)-\nu_t(E)\right|
=\left|\mathbb{E}\left(t f(W_{A,m}+1)-W_{A,m}f(W_{A,m})\right)\right|.
$$
Hence we can proceed as follows:
\begin{align}\label{errornewrepresentation}
\left|\mathbb{P}(W_{A,m}\in E)-\nu_t(E)\right|
 &=\left|t\mathbb{E}f(W_{A,m}+1)-\mathbb{E}\left(\sum_{i=1}^{m}I_if(W_{A,m})\right) \right|\notag\\  
&=\left|\sum_{i=1}^{m}p_i\mathbb{E}f(W_{A,m}+1)-\sum_{i=1}^{m}p_i\mathbb{E}(f(W_{A,m})|I_i=1)\right|\notag\\
&=\sum_{i=1}^{m}p_i\left(\sum_{a=0}^{m}f(a+1)\mathbb{P}(W_{A,m}=a)-\sum_{a=0}^{m}f(a)\mathbb{P}(W_{A,m}=a|I_i=1)\right)\notag\\
&= \sum_{i=1}^{m}p_i\sum_{a=0}^{m}f(a+1)\epsilon_{a,i},
\end{align} 
where we put $I_i(x)=\chi_AT^i(x)$ for the characteristic function of the set $T^{-i}A$ and
\begin{equation}\label{errorterm1}
\epsilon_{a,i}=\left|\mathbb{P}(W_{A,m}=a)-\mathbb{P}(W_{A,m}=a+1|I_i=1)\right|.
\end{equation}
The function $f$ above is the solution of the Stein equation~$(\ref{steineq})$ that 
corresponds to the indicator function $h=\chi_E$ in the Stein method and has been 
bounded in Lemma~$\ref{logsum}$.

In order to estimate the error term $\epsilon_{a,i}$  put 
$W_{A,m}^i=W_{A,m}-\chi_A\circ T^i=\sum_{\substack{1\le j\le m \\ j\neq i}}\chi_A\circ T^j$
(punctured sum). Then 
$$
\epsilon_{a,i}
=\left|\mathbb{P}(W_{A,m}=a)-\frac{\mathbb{P}\left(\{W_{A,m}^i=a\}\cap T^{-i}A\right)}{\mu(A)}\right|
\le\left|\mathbb{P}(W_{A,m}=a)-\mathbb{P}(W_{A,m}^i=a)\right| +\frac{\xi_a}{\mu(A)} 
$$
where 
 $\xi_a=\max_i\left|\mathbb{P}(\{W_{A,m}^i=a\}\cap T^{-i}A)-\mathbb{P}(W_{A,m}^i=a) \mu(A)\right|$
 is zero if all $I_i$ are independent of each other.
The first term is estimated by
$$
\left|\mathbb{P}(W_{A,m}=a)-\mathbb{P}(W_{A,m}^i=a)\right|\le \mathbb{P}(I_i=1)=\mu(A).
$$
For the second term, which contains $\xi_a$, we proceed as follows.

Let $\Delta<\!\!<m$ be a positive integer (the halfwith of the gap) and put for every $i\in(0,m]$
\begin{align}W_{A,m}^{i, -}&=\sum\limits_{j=1}^{i-(\Delta+1)}\chi_A\circ T^j, &
W_{A,m}^{i, +}&=\sum\limits_{j=i+\Delta+1}^{m}\chi_A\circ T^j, \nonumber\\
 U_m^{i, -}&=\sum\limits_{j=i-\Delta}^{i-1}\chi_A\circ T^j, &
U_m^{i,+}&=\sum\limits_{j=i+1}^{i+\Delta}\chi_A\circ T^j, \nonumber
 \end{align}
with the obvious modifications if $i<\Delta$ or $i>m-\Delta$.
With these  partial sums we distinguish between the hits that occur near  the $i^{th}$ iteration, namely $U_m^{i,-}$ and  $U_m^{i,+}$, and the hits that occur away from the $i^{th}$ iteration, namely
$ W_{A,m}^{i, -}$ and $ W_{A,m}^{i, +}$. Let us put $\tilde{W}_{A,m}^{i}= W_{A,m}^i-U_m^{i} =W_{A,m}^{i, -} + W_{A,m}^{i, +}$
for the total sum minus the $2\Delta+1$ terms in the gap surrounding the coordinate $i$.
The gap allows us to use the mixing property in the terms $W_{A,m}^{i,\pm}$
 and its size will be determined later when we optimise the error term. 

Note that for $a\in\mathbb N_0$
\begin{eqnarray*}
\mathbb{P}(\{W_{A,m}=a+1\}\cap T^{-i}A)&=&\mathbb{P}(\{W_{A,m}^i=a\}\cap T^{-i}A) \\
&=&\sum_{\substack{\vec{a}=(a^-,a^{0,-},a^{0,+},a^+)\\ \text{s.t } |\vec{a}|=a}}
\mathbb{P}\big(\{W_{A,m}^{i, \pm}=a^\pm\}\cap \{U_m^{i,\pm}=a^{0,\pm}\}\cap T^{-i}A \big).                                     
\end{eqnarray*}
We split the following sum into three terms
$$
\sum_a|f(a+1)|\cdot\bigg| \mathbb{P}\left(\{W_{A,m}^i=a\}\cap T^{-i}A\right)-\mathbb{P}\left(W_{A,m}^i=a\right)\mu(A)\bigg|\le R_1+R_2+R_3
$$
and will estimate the three terms
\begin{eqnarray*}
R_1&=&\sum_a|f(a+1)|\cdot\left|\mathbb{P}\left(\{W_{A,m}^i=a\}\cap T^{-i}A\right)-\mathbb{P}\left(\{\tilde{W}_{A,m}^i=a\}\cap T^{-i}A\right)\right|\\
R_2&=&\sum_a|f(a+1)|\cdot\left|\mathbb{P}\left(\{\tilde{W}_{A,m}^i=a\}\cap T^{-i}A\right)-\mathbb{P}\left(\tilde{W}_{A,m}^i=a\right)\mathbb{P}\left(I_i=1\right)\right|\\
R_3&=&\sum_a|f(a+1)|\cdot\left|\mathbb{P}\left(\tilde{W}_{A,m}^i=a\right)-\mathbb{P}\left(W_{A,m}^i=a\right)\right|\mu(A)
\end{eqnarray*}
separately.

\vspace{3mm}

\noindent\textbf{Estimate of $R_1$: } Observe that
\begin{eqnarray*}
\{W_{A,m}^i=a\}\cap T^{-i}A
&\subset& \left(\{\tilde{W}_{A,m}^i=a\}\cap T^{-i}A\right)\cup\left(\{U_m^i>0\}\cap T^{-i}A\right)\\
\{\tilde{W}_{A,m}^i=a\}\cap T^{-i}A&\subset& \left(\{W_{A,m}^i=a\}\cap T^{-i}A\right)
\cup\left(\{U_m^i>0\}\cap T^{-i}A\right).
\end{eqnarray*}
Since $U_m^i=U_m^{i,+}+U_m^{i,-}>0$ implies that either $U_m^{i,+}>0$ or $U_m^{i,-}>0$ we get
$$
\big|\mathbb{P}\big(\{W_{A,m}^i=a\}\cap T^{-i}A\big)-\mathbb{P}\big(\{\tilde{W}_{A,m}^i=a\}\cap T^{-i}A\big)\big|
\le \mathbb{P}\big( \{U_m^i>0\}\cap T^{-i}A\big)\le b^-_i+b^+_i
$$
where
$$
b^-_i=\mathbb{P}\big( \{U_m^{i,-}>0\}\cap T^{-i}A\big)\quad\text{and}\quad
b^+_i=\mathbb{P}\big( \{U_m^{i,+}>0\}\cap T^{-i}A\big).
$$
 For $b^+_i$ we obtain the estimate
$$
 b_i^+=\mathbb{P}\big( \{U_m^{i,+}>0\}\cap T^{-i}A\big)
= \mathbb{P}(U_m^{i,+}>0|I_{i}=1)\mu(A)
=\mathbb{P}_A(\tau_A\le\Delta)\mu(A)
$$
and since in~\cite{HP} it was shown that  $b^-_i=b^+_i$ we obtain
$$
R_1 \le c_2\mathbb{P}_A(\tau_A\le\Delta)\mu(A)\sum_a|f(a+1)| 
\le c_3\mathbb{P}_A(\tau_A\le\Delta)\mu(A)(t+\log m)
$$
for some $c_3$ where we used Lemma~\ref{logsum} to estimate the sum over $a$.

\vspace{3mm}

\noindent\textbf{Estimate of $R_3$: } In order to show that short returns are negligible
note that 
\begin{eqnarray*}
\{W_{A,m}^i=a\}&\subset &\{\tilde{W}_{A,m}^i=a\}\cup\{U_m^i>0\}\\
\{\tilde{W}_{A,m}^i=a\}& \subset&\{W_{A,m}^i=a\}\cup\{U_m^i>0\}
\end{eqnarray*}
which yields
$$
\bigg| \mathbb{P}\left(\tilde{W}_{A,m}^i=a\right)-\mathbb{P}\left(W_{A,m}^i=a\right) \bigg|
\le \mathbb{P}\left(U_m^i>0 \right)
\le 2\mathbb{P}\left( \bigcup_{k=1}^{\Delta}\{I_{i+k}=1\} \right)
\le 2\Delta \mu(A),
$$
and therefore
$$
R_3\le 2\Delta \mu(A)^2\sum_a|f(a+1)|\le c_4\Delta \mu(A)^2(t+\log m).
$$

\vspace{3mm}

\noindent\textbf{Estimate of $R_2$: } This is the principal term and the speed of mixing now becomes
relevant.
Recall that $ \tilde{W}_{A,m}^{i}(x)=W_{A,m}^{i, -}(x) + W_{A,m}^{i, +}(x)$ and
we want to estimate 
\begin{eqnarray*}
R_2&\le&\sum_{a}|f(a+1)|\bigg|\mathbb{P}\left(\{\tilde{W}_{A,m}^i=a\}\cap T^{-i}A\right)
-\mathbb{P}\left(\tilde{W}_{A,m}^i=a\right)\mu(A)\bigg|\\
&\le&\sum_{a^-,a^+}|f(a^-+a^++1)|\left(\mathbb{P}\left(\{\tilde{W}_m^{i,\pm}=a^\pm\}\cap T^{-i}A\right)
-\mathbb{P}\left(\tilde{W}_m^{i,\pm}=a^\pm\right)\mu(A)\right)\epsilon_{a^-,a^+},
\end{eqnarray*}
where $\epsilon_{a^-,a^+}=\sgn\left(\mathbb{P}\left(\{\tilde{W}_{A,m}^i=a\}\cap T^{-i}A\right)
-\mathbb{P}\left(\tilde{W}_{A,m}^i=a\right)\mu(A)\right)$.
If we put 
$$
\mathcal{W}^+(a^-)=\bigcup_{a^+:\,\epsilon_{a^-,a^+}=+1}\{\tilde{W}^{i,+}_m=a^+\},\qquad
\mathcal{W}^-(a^-)=\bigcup_{a^+:\,\epsilon_{a^-,a^+}=-1}\{\tilde{W}^{i,+}_m=a^+\},
$$
both disjoint unions, then
\begin{eqnarray*}
R_2&\le&\sum_{a}|\varphi(a)|\bigg|\mathbb{P}\left(\{\tilde{W}_m^{i,-}=a^-\}\cap\mathcal{W}^+(a^-)
\cap T^{-i}A\right)
-\mathbb{P}\left(\tilde{W}_m^{i,-}=a^-\right)\mu\left(\mathcal{W}^+(a^-)\right)\mu(A)\bigg|\\
&&+\sum_{a}|\varphi(a)|\bigg|\mathbb{P}\left(\{\tilde{W}_m^{i,+}=a^+\}\cap\mathcal{W}^-(a^+)
\cap T^{-i}A\right)
-\mathbb{P}\left(\tilde{W}_m^{i,+}=a^+\right)\mu\left(\mathcal{W}^-(a^+)\right)\mu(A)\bigg|
\end{eqnarray*}
where $\varphi(a)=\sup_{a'>a}|f(a')|$ satisfies by Lemma~\ref{logsum} $\varphi(a)\le\min(1,\frac{t}a)$.
We gave to estimate the two mixing terms, the first of which is for $a^-\ge0$:
$$
\bigg|\mathbb{P}\left(\{\tilde{W}_m^{i,-}=a^-\}\cap\mathcal{W}^+(a^-)\cap T^{-i}A\right)
-\mathbb{P}\left(\tilde{W}_m^{i,-}=a^-\right)\mu\left(\mathcal{W}^+(a^-)\right)\mu(A)\bigg|
\le R_{2,1}+R_{2,2}+R_{2,3}
$$
where
\begin{eqnarray*}
R_{2,1}&=& \bigg|\mathbb{P}\left(\{\tilde{W}_m^{i,-}=a^-\}\cap\mathcal{W}^+(a^-)\cap T^{-i}A\right)
-\mathbb{P}\left(\{\tilde{W}_m^{i,-}=a^-\}\cap T^{-i}A\right)\mu\left(\mathcal{W}^+(a^-)\right)\bigg|\\
R_{2,2}&=&\bigg|\mathbb{P}\left(\{W_{A,m}^{i,-}=a^-\}\cap T^{-i}A\right)-
\mathbb{P}\left(W_{A,m}^{i,-}=a^-\right)\mu(A)\bigg|\mu\left(\mathcal{W}^+(a^-)\right)\\
R_{2,3}&=&\bigg|\mathbb{P}\left(W_{A,m}^{i,-}=a^-\right)\mu\left(\mathcal{W}^+(a^-)\right)-
\mathbb{P}\left(\{\tilde{W}_m^{i,-}=a^-\}\cap\mathcal{W}^+(a^-)\right)\bigg|\mu(A).
\end{eqnarray*}
We now bound the three terms separately: Due to the mixing property we get for the first term
the estimate
$$
R_{2,1}\le \alpha(\Delta).
$$
Similarly for the second term 
$$
R_{2,2}\le\alpha(\Delta)\mu\left(\mathcal{W}^+(a^-)\right),
$$
while the third term is estimated by
$$
R_{2,3}\le\alpha(2\Delta)\mu(A).
$$
Combining these estimates and considering that the second term in
the above estimate of $R_2$ is estimated in the same manner we obtain
$$
R_2\le c_4\alpha(\Delta)\sum_{a}\varphi(a)\le c_6\alpha(\Delta)(t+\log m)
$$
for some constant $c_4$.

\vspace{3mm}

\noindent Finally, putting together the error terms $R_1$, $R_2$ and $R_3$  yields
\begin{align*}
&\left|\mathbb{P}(W_{A,m}\in E)-\nu_t(E)\right|\\
&\le\sum_{i=1}^m p_i\left(\sum_{a=0}^m|f(a+1)|\mu(A)+c_7\left(\frac{\alpha(\Delta)}{\mu(A)}
+2\Delta\mu(A)+\mathbb{P}_A(\tau_A\le\Delta)\right)(t+\log m)\right)\\
&\le c_8\left(\mu(A)+\frac{\alpha(\Delta)}{\mu(A)}+\Delta\mu(A) +  \mathbb{P}_A(\tau_A\le\Delta)  \}\right)(t+\log m)
\end{align*}
for some $c_8$ independent of $A$.
\end{proof}

\section{Poisson distributed return times for Bowen  balls }\label{bowen.poisson}

In this section we will prove Theorem~\ref{t1}. Recall that
$$
\psi(\epsilon , \delta, x) = \frac{\mu(B(x,\epsilon + \delta) \setminus B(x, \epsilon - \delta))}{\mu(B(x,\epsilon))}
$$
is the proportion  of the measure of the annulus  $B(x, \epsilon  + \delta) \setminus B(x, \epsilon - \delta)$ to the ball $B(x, \epsilon )$. 
Put $\tau^k_A(x)=\tau_A\circ T^{\tau_A^{k-1}}$ for the $k$th return of $x$ to the set $A$:
$$
\tau^k_A(x) = \min\{k > \tau^{k-1}_A(x): T^k(x) \in A\}
$$
where $\tau^1_A = \tau_A$.

We will prove the following more general theorem and then deduce Theorem~\ref{t1} and Theorem~\ref{t1'}. 

\begin{thm}\label{t5}
Let $\mu$ be a $\phi$-mixing $T$-invariant ergodic measure with positive entropy. Let $\gamma_n = \diam(\mathcal{A}^n)$. Assume that there exist $\epsilon_0>0$ and an increasing sequence $\{N(n)\}_{n=1}^\infty$ satisfying $n < N(n) < \frac{1}{4}\mu(\bna)^{-1}$ such that 
\begin{equation}\label{t5e}
\psi(\epsilon, \gamma_{N(n) -k}, T^kx) \le \vartheta_n(\epsilon)\cdot \frac{\mu(\bna)}{n}
\end{equation}
for all $\epsilon < \epsilon_0, x\in X, 0 \le k < n$, where $\vartheta_n(\epsilon) \to 0$ as $n \to \infty$
($\forall\;\epsilon < \epsilon_0$).

Then for all $t>0$ one has
$$
\lim\limits_{n \to \infty}\mathbb{P}\left(W_{\bna, m} = k\right) =e^{-t} \frac{t^k}{k!},
$$
where  $m = \frac{t}{\mu(\bna)}$.
\end{thm}

The idea of the proof is to use cluster of cylinders sets to approximate Bowen balls. For this purpose, for some integer $N(n) \gg n$, define 
$$
\bnb =  \bigcup\limits_{A \in \mathcal{A}^{N(n)}, A \subset \bna } A
$$
the union of all $N(n)$-cylinders contained in $\bna$. If we put 
$$
\bnc =  \bigcup\limits_{A \in \mathcal{A}^{N(n)}, A \cap \partial\bna \ne \emptyset } A
$$
as the union of all cylinders which intersect the boundary of $\bna$, then 
$$
\bna \setminus \bnb \subset \bnc .
$$

The next lemma (c.f.~\cite{HY}) allows us to estimate the difference between $\bnb$ and $\bna$.

\begin{lemma}\label{l5} Under the hypothesis of Theorem~\ref{t5} we have
$$\mu(\bnc) \le  \vartheta_n(\epsilon) \mu(\bna)$$
and in particular,  $\mu(\bna)/\mu(\bnb) = \mathcal{O}(1)$.
\end{lemma}
\begin{proof}
Since $T$ is continuous,  $\partial \bna \subset \bigcup\limits_{k=0}^{n-1}T^{-k}\partial B(T^kx, \epsilon)$. Hence if $A_{N(n)} \cap \partial\bna \ne \emptyset$ for some $N(n)$-cylinder $A_{N(n)}$,
 then $A_{N(n)-k}(T^ky) \cap \partial B(T^kx,\epsilon) \ne \emptyset$ for some $0 \le k \le n-1$
 and  $y \in A^{N(n)}$. Since $\mbox{diam}(A_{N(n)-k}(T^ky)) \le \gamma_{N(n)-k}$ we obtain
\begin{align*}
\bnc \subset & \bigcup\limits_{k=0}^{n-1} 
T^{-k}(B(\partial B(T^kx, \epsilon), \gamma_{N(n)-n}))\\
\subset& \bigcup_k 
T^{-k}(B(T^kx, \epsilon+ \gamma_{N(n)-k}) \setminus B(T^kx, \epsilon + \gamma_{N(n)-k})),
\end{align*}
and consequently
\begin{align*}
\mu(\bnc) \le& n \cdot\sup_{0 \le k \le n-1} 
\mu(B(T^kx, \epsilon+ \gamma_{N(n)-k}) \setminus B(T^kx, \epsilon + \gamma_{N(n)-k}))\\
=& n \cdot \sup_{0 \le k \le n-1} \{\psi(\epsilon, \gamma_{N(n)-k}, T^kx)\cdot \mu(B(T^kx, \epsilon))\}\\
\le &  n \cdot \sup_{0 \le k \le n-1} \{\psi(\epsilon, \gamma_{N(n)-k}, T^kx)\}\\
\le&  \vartheta_n(\epsilon) \mu(\bna).
\end{align*}
In particular  $\mu(\bnc)/\mu(\bnb) = \vartheta_n \to 0$ and therefore 
$\mu(\bna)/\mu(\bnb) = \mathcal{O}(1)$.
\end{proof}

Next we show that the limiting distribution for the hitting times of $\bna$ can be approximated by 
the distribution of $\bnb$. To simplify notation we write $B = \bna$ and $\tb = \bnb$. For $t > 0$ 
we put $m = \frac{t}{\mu(B)}$ and $\tm = \frac{t}{\mu(\tb)}$ and write 
for simplicity's sake 
$$
\Theta_{B,m}(k)=\mathbb{P}\Big(W_{B,m} =k \Big),\quad 
\Theta_{\tb,\tm} (k)= \mathbb{P}\left( W_{\tb,\tm} =k \right)
$$
and others similarly. The following approximation lemma 
does note depend on the mixing property. 

\begin{lemma}\label{l13}  For all $t \ge 0$ we have
$$
 \left| \Theta_{B,m}(k)  - \Theta_{\tb,\tm} (k)\right| \le 2t \cdot \vartheta_n(\epsilon) \to 0
$$
as $n \rightarrow \infty$.
\end{lemma}

\begin{proof} By the triangle inequality
\begin{eqnarray*}
\left| \Theta_{B,m}(k) - \Theta_{\tb,\tm} (k) \right|
&\le & \left| \Theta_{B,m}(k)   -\Theta_{\tb,m}(k)\right|
+\left|\Theta_{\tb,m}(k) -  \Theta_{\tb,\tm} (k) \right|\\
&=& I + II.
\end{eqnarray*}
In order to estimate the first term note that $\tb \subset B$ which implies $W_{B,m} \ge W_{\tb,m}$. Consequently
 $$
 I \le \mathbb{P}(W_{B \setminus \tilde B, m}>0)
 \le \mathbb{P}(\tau_{B\setminus \tilde B} < m)
 \le m\mu(B\setminus \tilde B).
 $$
For the second term we proceed as follows:
$$
II = \mathbb{P}(\{W_{\tilde B, m} = k\} \cap \{W_{\tilde B, \tm} >  k\} ) 
 \le \mu(\tilde B) (\tm-m)
= m\mu(B \setminus \tilde B).
$$

Combining the estimates for $I$ and $II$ yields by Lemma~\ref{l5}
\begin{eqnarray*}
\left|  \Theta_{B,m}(k)   -\Theta_{\tb,\tm} (k) \right|
& \le &2m \cdot \mu(B \setminus \tilde B)\\
&\le & 2m \cdot \mu(\bnc)\\
& \le & 2m \cdot \vartheta_n(\epsilon) \mu(B)\\
 &= &2t \vartheta_n(\epsilon) \to 0.
\end{eqnarray*}
\end{proof}

Before we prove Theorem~\ref{t5} let us consider the case of $\alpha$-mixing measures.
As noted in~\cite{HY} generalised SRB measures for systems that allow a Young tower
construction as in~\cite{Y2,Y3} are $\alpha$-mixing and thus are prime examples to which
the following proposition can be applied.
We though have to make an assumption on the short retun times.

\begin{proposition}\label{t5alpha}
Let $\mu$ be an $\alpha$-mixing measure where $\alpha(k)$ is decreases exponentially
fast to $0$. Let $\Delta=a\left|\log\mu(\bna)\right|$ where $a>0$ is so that 
$\frac{\alpha(\Delta)}{\mu(\bna)}\Delta\to0$ as $n\to\infty$. 
If $\mathbb{P}_{\bna}(\tau_{\bna}\le\Delta)\Delta\to0$ then
$$
\lim\limits_{n \to \infty}\mathbb{P}\left(W_{\bna, m} = k\right) =e^{-t} \frac{t^k}{k!}
$$
\end{proposition}

\begin{proof} By Lemma~\ref{l13} it is sufficient to prove that
$$
\lim\limits_{n \rightarrow \infty} \Theta_{\tb,\tm}(k) =e^{-t}\frac{t^k}{k!}.
$$
The result the follows from from Theorem~\ref{Poisson_cylinder} with $m=1/\mu(\bna)$.
\end{proof}






Let us now prove Theorem~\ref{t5} where the $\phi$-mixing property is used 
 to control the short return times up to $\Delta$.


\begin{proof}[Proof of Theorem~\ref{t5}]
Again, by Lemma~\ref{l13} it is enough to show that 
$\Theta_{\tb,\tm}(k) \to e^{-t}\frac{t^k}{k!}$ as $n\to\infty$.
We apply Theorem~1 of~\cite{HP} to the set $\bnb \in\sigma( \mathcal{A}^{N(n)})$
and obtain (for some $c_1$)
$$
\left|\Theta_{\tb,\tm}(k)- e^{-t}\frac{t^k}{k!} \right|\\ 
\le c_1 t(t\vee 1) \inf_{\Delta>0}\{\Delta\mu(\tilde B) + \sum\limits_{j=\tau({\tilde B})}^{\Delta}\delta_{\tilde B} (j) + \frac{\phi(\Delta)}{\mu(\tilde B)} \}|\log \mu(\tilde B)|,
$$
where $\delta_{\tb}(j) =  \min\limits_{1 \le \omega \le j \wedge N(n)} \{ \mu(A_\omega (\tb))  + \phi(j-\omega)\} $ and, as before, $A_\omega(\tb) = \bigcup\limits_{A\in\mathcal{A}^\omega, A\cap \tb \ne \emptyset}A$.
Let $\eta' \in(\frac{1}{2+\kappa},1) $ so that the gaps $\Delta = \mu(\tilde B)^{-\eta'}$ are 
larger than $N(n)$. Then
$$
\left|\Theta_{\tb,\tm}(k)- e^{-t}\frac{t^k}{k!} \right|
\le c_1 t(t\vee 1) \left(\mu(\tilde B)^{1-\eta'} + \sum\limits_{j=\tau({\tb})}^{\Delta}\delta_{\tilde B} (j) + \mu(\tilde B)^{-1+\eta'(2+\kappa)} \right)|\log \mu(\tilde B)|.
$$
Since $\mu(\tilde B) = \mathcal{O}( \mu(\bna))$ we conclude by~\cite{BK}  
$|\log\mu(\tilde B)| = \mathcal{O}(n)$
and it thus remains to show that
 $\sum\limits_{j=\tau({\tilde B})}^{\Delta}\delta_{\tilde B} (j)  = o(\frac{1}{n})$. 
 
 Since $\tilde B \subset B$ we get $A_\omega(\tilde B)\subset A_\omega(B)$ 
 and therefore $A_\omega(B) \subset \bnb \cup \bnc$ for all $\omega \ge N(n)$. 
 As in Lemma~\ref{cluster of cylinders} let us put
 $$
 D_{N_0} = \{x \in X: \frac{1}{n} \sum\limits_{k=0}^{n-1} \chi_{U_\epsilon(\mathcal{A})}(T^k(x)) < \beta  
 \quad \forall n \geq N_0     \}.
 $$
 Then $A_\omega(\bna) \subset A_n(\bna) \subset \cn$ for all $x \in D_{N_0}$ and $\omega \ge n \ge N_0$.

By Theorem~\ref{period}, we can take $N_0$ large enough such that the set 
$$
\{x: \tau(\bna) > \frac{n}{2} \quad \forall n > N_0 \}
$$  
has measure arbitrarily close to $1$. Since $\bnb \subset \bna$ we conclude that  
$$
E_{N_0} = \{x: \tau(\bnb) > \frac{n}{2} \quad\forall n > N_0\} 
$$
also has measure arbitrarily close to $1$ for $N_0$ large enough.  For
 $x\in G_{N_0} = D_{N_0} \cap E_{N_0}$, and all $n > 4N_0$ we then split the following sum into 
 three parts:
\begin{align*}
\sum\limits_{j=\tau({\tilde B})}^{\Delta}\delta_{\tilde B} (j) = & \sum\limits_{j=\tau({\tilde B})}^{\Delta}\min\limits_{1 \le \omega \le j\wedge N(n)}\{\mu(A_\omega (\tilde B))+\phi(j - \omega)\}\\
\le &  \sum\limits_{j=n/2}^{\Delta}\min\limits_{1 \le \omega \le j\wedge N(n)}\{\mu(A_\omega (B))+\phi(j - \omega)\}\\
\le   &   \sum\limits_{j=n/2}^{2n-1}\min\limits_{1 \le \omega \le j}\{\mu(A_\omega (B))+\phi(j - \omega)\}
+\sum\limits_{j=2n}^{N(n)}\min\limits_{1 \le \omega \le j}\{\mu(A_\omega (B))+\phi(j - \omega)\} \\
&\hspace{4cm}+ \sum\limits_{j=N(n)+1}^{\Delta}\min\limits_{1 \le \omega \le  N(n)}\{\mu(A_\omega (B))+\phi(j - \omega)\}\\
=   & I + II+III.
\end{align*}
Since $\mu$ is $\phi$-mixing there exists a $\nu<1$ so that $\mu(A_m(x))<\nu^m$ for all $x$ 
and $m$ large enough~\cite{Abadi01}. We now assume that $\beta>0$ is small enough 
so that the size $\lambda_m$
of the  $(\beta,m)$-clusters $\mathcal{C}_{\beta,m}(x)$ satisfies $\lambda_m<\nu^{-\frac{m}2}$
for all $x\in D_{N_0}$
 (see~\eqref{lambda.estimate}).
Thus 
\begin{equation}\label{cluster.estimate}
\mu(\mathcal{C}_{\beta,m}(x)) \le \lambda_m\nu^m<\nu^\frac{m}2
\end{equation}
for all $m$ large enough and $x\in D_{N_0}$. We now estimate the three parts on the RHS above as follows:\\
(I) For the term $I$, we also take $\omega=\frac{j}{2}$. Since $\bna \subset B_{\epsilon, \frac{n}{4}}(x)$ and $\frac{j}{2} \ge \frac{n}{4} \ge N_0$ we have 
$$
A_\frac{j}{2}(\bna) \subset A_\frac{j}{2}( B_{\epsilon, \frac{n}{4}}(x)) \subset A_{\frac{n}{4}}(B_{\epsilon, \frac{n}{4}}(x)) \subset \mathcal{C}_{\beta, \frac{n}{4}}(x).
$$
The bound~\eqref{cluster.estimate} then yields
$$
I \le\sum\limits_{j=n/2}^{2n-1}\mu(A_{\frac{j}{2}} (B))+\phi(\frac{j}{2})
 \le 2n \mu(\mathcal{C}_{\beta, \frac{n}{4}}(x)) + \frac{c_2}{n^{1+\kappa}} = o(\frac{1}{n}).
$$
(II) For the second term we take $\omega = \frac{j}{2}$  and obtain
$$
II \le  \sum\limits_{j=2n}^{N(n)}\left(\mu(A_{j/2} (B))+\phi(\frac{j}{2})\right)
\le  \sum\limits_{j=2n}^{N(n)} \mu(\mathcal{C}_{\beta, \frac{j}{2}}(x)) + \frac{c_1}{n^{1+\kappa}}
$$
since $\frac{j}{2} \ge n > N_0$. By~\eqref{cluster.estimate} we conclude  that $II = o(\frac{1}{n})$.\\
(III) For the third term $III$ we take $\omega = \frac{N(n)}{2}$. Lemma~\ref{l5} shows that $\mu(A_{N(n)/2} (B) )= \mathcal{O}(1)\mu(B)$. We obtain
\begin{align*}
III \le & \sum\limits_{j=N(n)+1}^{\Delta}\left(\mu(A_{N(n)/2} (B))+\phi(j - \frac{N(n)}{2})\right)\\
\le & \sum\limits_{j=N(n)+1}^{\Delta}\mathcal{O}(1)\mu(\tilde B) + \frac{c_2}{N(n)^{1+\kappa}}\\
= & \mathcal{O}(1) \Delta\mu(\tilde B)+ \frac{c_2}{N(n)^{1+\kappa}}\\
= &o(\frac{1}{n}).
\end{align*}
The three estimates combined prove Theorem~\ref{t5}.
 \end{proof}

To prove Theorem~\ref{t1} and \ref{t1'} we  need to verify that (\ref{t5e}) is satisfied.

\begin{proof}[Proof of Theorem~\ref{t1}.]Under the hypothesis of Theorem~\ref{t1} we take $\eta\in (\frac{1}{\xi \zeta},1)$ and put $N(n) = \mu(\bna)^{-\eta}$. This yields
\begin{align*}
\frac{n \cdot\psi(\epsilon,\gamma_{N(n)-k},T^kx)}{\mu(\bna)}& \le \frac{n C_\epsilon}{N(n)^{\xi \cdot  \zeta}|\log\gamma|^{\zeta}\mu(\bna)}\\
& = C_\epsilon' n \mu(\bna)^{\eta\xi\zeta-1} \to 0
\end{align*}
since $\eta \xi \zeta >1$. Consequently $\vartheta_n(\epsilon)\to0$ as $n\to\infty$ for every 
small enough $\epsilon>0$ and the statement of Theorem~\ref{t1} now follows from
Theorem~\ref{t5}
\end{proof}

\begin{proof}[Proof of Theorem~\ref{t1'}.] Similarly the choice of $\eta \in (\frac{1}{\alpha\xi},1)$ 
and $N(n) = \mu(\bna)^{-\eta}$  yield
$$
\frac{n \cdot\psi(\epsilon,\gamma_{N(n)-k},T^kx)}{\mu(\bna)} \le \frac{n C_\epsilon N(n)^{-\alpha\xi}}{\mu(\bna)}
 = C_\epsilon' n \mu(\bna)^{\eta\alpha\xi-1} \to 0
$$
since $\eta\alpha\xi>1$. Again $\vartheta_n(\epsilon)\to0$ as $n\to\infty$ for every 
small enough $\epsilon>0$ and the theorem follows from Theorem~\ref{t5}.
\end{proof}







\begin{thebibliography}{99}

\bibitem{Abadi01} M Abadi: Exponential Approximation for
Hitting Times in Mixing Stochastic Processes; Mathematical Physics Electronic
Journal~7 (2001)

\bibitem{Abadi08} M Abadi: Poisson approximations via Chen-Stein for non-Markov
processes; {\it In and Out of Equilibrium 2} V Sidoravicius and M E Vares (editors),
2008, 1--19.

\bibitem{AV3} M Abadi  and N Vergne: Poisson approximation for search of rare words in DNA sequences; ALEA-Lat.\ Am.\ J. Prob.\ Math.\ Stat.\ {\bf 4}, 233--244.

\bibitem{AV1} M Abadi and N Vergne: Sharp errors for point-wise Poisson approximations 
in mixing processes; Nonlinearity {\bf 21} (2008), 2871--2885.

\bibitem{BC1} A D Barbour and L H Y Chen: {\it An Introduction to Stein's Method}; Lecture Notes Series, Institute for Mathematical Sciences, National University of Singapore, Vol.~{\bf 4}, 2005.

\bibitem{BK} M Brin and A Katok: On local entropy; Proceedings, Springer.

\bibitem{CC13} J-R Chazottes and P Collet: Poisson approximation for the number of visits to balls in nonuniformly hyperbolic dynamical systems; Ergod.\ Th.\ \& Dynam.\ Sys.\ {\bf 33} (2013), 49--80.

\bibitem{Den} M Denker: Remarks on weak limit laws for fractal sets;
{\it Progress in Probability} Vol.\ 37, Birkh\"auser 1995, 167--178.

\bibitem{DGS} M Denker, M Gordin and A Sharova: A Poisson limit theorem for toral automorphisms;
Illinois J. Math.\ {\bf 48(1)} (2004), 1--20.

\bibitem{Doe} W Doeblin: Remarques sur la th\'eorie m\'etrique des fraction continues; 
Compositio Mathematica {\bf 7} (1940), 353--371.

 
 \bibitem{H13} N T A Haydn:  Entry and return times distribution; 
Dynamical Systems: An International Journal {\bf 28(3)} (2013), 333--353.

\bibitem{HLV}
N Haydn, Y Lacroix and S. Vaienti: Hitting and Return Times in Ergodic Dynamical Systems;
Ann.\ of Probab.\ {\bf 33} (2005), 2043--2050.

\bibitem{HP} N Haydn and Y Psiloyenis: Return times distribution for Markov towers with 
decay of correlations; Nonlinearity {\bf 27} (2014) 1323--1349

\bibitem{HW} N Haydn and K Wasilewska: Limiting distribution and error terms for the number of visits to balls in non-uniformly hyperbolic dynamical systems;
  available at {\tt http://arxiv.org/abs/1402.2990}.

\bibitem{HY} N Haydn and F Yang: Entry times distribution for dynamical balls on metric spaces; available at {\tt http://arxiv.org/abs/1410.8640}

\bibitem{Hirata2} M Hirata: Poisson law for the dynamical systems with the 
``self-mixing'' conditions; {\em Dynamical Systems and Chaos}, Vol.~1
(Worlds Sci.\ Publishing, River Edge, New York (1995), 87--96.

\bibitem{Kif12} Y Kifer and A Rapaport: Poisson and compound Poisson 
approximations in a nonconventional setup; preprint 2012,
available at {\tt http://arxiv.org/abs/1211.5238}.
  
  \bibitem{KL} M Kupsa and Y Lacroix: Asymptotics for hitting times, Ann.\ of Probab.\ {\bf 33(3)} (2005), 610--614.

\bibitem{L02} Y Lacroix: Possible limit laws for entrance times of an ergodic aperiodic dynamical system;
 Israel J.\ Math.\ {\bf 132} (2002), 253--263.



\bibitem{PS} F P\`ene and B Saussol: Poisson law for some 
nonuniformly hyperbolic dynamical systems with polynomial rate of mixing;
preprint Universit\'e de Bretagne Occidentale

 \bibitem{Pit} B Pitskel: Poisson law for Markov chains;
Ergod.\ Th.\  \& Dynam.\ Syst.\ {\bf 11} (1991), 501--513.

\bibitem{STV} 
B Saussol, S Troubetzkoy and S Vaienti,  Recurrence,
dimensions and Lyapunov exponents, {\em J. Stat.\ Phys.}, \textbf{106}
(2002), 623--634.

\bibitem{Var} P Varandas: Entropy and Poincar\'{e} recurrence from a geometrical viewpoint; 
Nonlinearity {\bf 22(10)} (2009), 2365--2375.

\bibitem{Wal} P Walters: {\it An Introduction to Ergodic Theory}; Springer-Verlag 1981.

\bibitem{WTW} H Wang, M Tang and R Wang: A Poisson limit theorem for
a strongly ergodic non-homogeneous Markov chain; J. Math.\ Analysis Applications 
{\bf 277} (2003), 722--730.

\bibitem{Y2}  L-S Young: Statistical properties of dynamical systems with some hyperbolicity;
Annals of Math.~{\bf 7} (1998), 585--650.

\bibitem{Y3}  L-S Young: Recurrence time and rate of mixing; Israel J. of Math.~{\bf 110} (1999),
  153--188.



\end{thebibliography}
\end{document}